\documentclass[11pt]{article}

\usepackage{amsmath}
\usepackage{amsfonts}
\usepackage{amssymb}
\usepackage{amscd}
\usepackage{amsthm}
\usepackage{latexsym}

\usepackage{graphics}
\usepackage{pb-diagram,lamsarrow,pb-lams}

\usepackage{mathrsfs}
\usepackage{epsfig}
\usepackage{color}

\usepackage{xy}

 \def\X{\mathcal X} \def\C{\mathbb{C}}

  \def\H{\mathcal H} 
 
\def\S{\mathcal S}

    \def\e{\mathfrak e}

\def\I{{\rm 1\kern-.26em I}}

\def\Op{\mathfrak{Op}^A}

\def\1{\mathfrak{1}}
\def\0{\mathfrak{0}}

 \def\hb{\hbar}

\def\({\left(}
\def\){\right)}
\def\[{\left[}
\def\]{\right]}
\def\<{\left<}
\def\>{\right>}

\newtheorem{lemma}{Lemma}[section]
\newtheorem{corollary}[lemma]{Corollary}
\newtheorem{theorem}[lemma]{Theorem}
\newtheorem{proposition}[lemma]{Proposition}
\newtheorem{definition}[lemma]{Definition}

\numberwithin{equation}{section}

\begin{document}

\title{The Modulation Mapping for Magnetic Symbols and Operators}

\date{\today}

\author{Marius M\u antoiu \footnote{Universidad de Chile, Las Palmeras 3425, Casilla 653, Santiago Chile.
Email: Marius.Mantoiu@imar.ro} and Radu Purice \footnote
{Institute of Mathematics Simion Stoilow of the Romanian Academy, P.O. Box 1-764, Bucharest, RO-70700, Romania.
Email: Radu.Purice@imar.ro}}

\maketitle \abstract{We extend the Bargmann transform to the
magnetic pseudodifferential calculus, using gauge-covariant
families of coherent states. We also introduce modulation
mappings, a first step towards adapting modulation spaces to the
magnetic case.} 

%...............................................................................................
\section*{Introduction}
%................................................................................................

Recent publications \cite{KO1,MP1,MP2,MPR1,IMP1,BB} introduced and developed a mathematical formalism for the
quantization of physical systems with variable magnetic fields. We would like now to complete the picture, sketching
the connection with coherent states, the Bargmann transform and a suitable version of the modulation mapping.

Classically, the magnetic field changes the geometry of the phase-space. This is realized by a modification of the standard
symplectic form and, consequently, of the Poisson algebra structure of the smooth functions on phase-space,
interpreted as classical observables.
Correspondingly, at the quantum level, one introduces \cite{KO1,MP1,MP3} algebras of observables defined only in terms of
the magnetic field, no choice of a vector potential being needed. The main new object is a
composition law on symbols defined by fluxes of the magnetic field through triangles.

To get self-adjoint operators and a Hilbert space theory, the resulting algebras are represented in Hilbert spaces;
this is done by choosing vector potentials defining the magnetic field. In such a way one gets essentially
a new pseudodifferential calculus (\cite{MP1,IMP1,IMP'}), seen as a functional calculus for
the family of non-commuting operators composed
of positions and magnetic momenta. When no magnetic field is present, it coincides with the Weyl quantization.
One of its main virtue is gauge-covariance: equivalent choices of vector potentials
lead to unitarily equivalent representations. We stress that this property is not shared by doing a minimal coupling
modification of the symbol in the usual Weyl calculus.

Both the intrinsic and the represented version admit $C^*$-algebraic reformulations (\cite{MPR1,MP1}). They were
useful in the spectral analysis of magnetic Schr\"odinger operators, cf. \cite{MPR2,LMR} for instance.

One of the purposes of this article is to define and study a modulation mapping in the setting of the magnetic
quantum formalism.
The main application of this magnetic modulation mapping would be inducing useful new function spaces on $\Xi$
from known function spaces on $\Xi\times\Xi$. This will be done in a future publication.

Modulation spaces are Banach function spaces introduced long ago by H. Feichtinger \cite{F1,F2}. By definition, they
involve norm estimates on a certain family of transformations of the function one studies, defined on
$\mathbb R^n$ or on a locally compact abelian group.
Modulation spaces evolved especially in connection with Time Frequency Analysis, Gabor Frames and Signal Processing
Theory. In \cite{S1}, J. Sj\"ostrand discovered the importance of one of these spaces in the theory of
pseudodifferential operators, cf. also \cite{S2}. Then the interconnection between modulation spaces and
pseudodifferential theory developed considerably, as a result both of "the Vienna school" and other researchers. We
cite, without any claim of completeness, \cite{BB1,B1,FGL,G1,G2,G3,G4,GT,HRT,T1}. Other important works are cited in these
articles.

To define modulation spaces, one introduces first a transformation (the Short Time Fourier Transform)
from functions defined on the phase space $\Xi=\mathbb R^{2N}$ to functions defined on
$\Xi\times\Xi$; we are going to indicate in Section \ref{sashaa} a magnetic analog of this transformation.
Since we are mainly interested in
its behavior with respect to symbol composition, we deviate to a certain extent from the standard approach;
so our definition could have some interest even in the non-magnetic case.
We show that this transformation is isometric between $L^2$-spaces and it transforms the
magnetic analog of the Weyl composition law into the multiplication in a typical crossed product algebra, which can also be seen
as the Kohn-Nirenberg composition for symbols defined in $\Xi\times\Xi$.

In Section \ref{finalaa} we show that our modulation mapping is an
intrinsic counter-part of the transformation sending operators from the Schr\"odinger representation to the
magnetic Bargmann representation. This one is induced by a proper choice of
a family of coherent states; for the standard case as well as for many generalizations we refer to \cite{AAGM,Fo,Ha,La3}
and to the references therein. Our result says that magnetic Weyl operators can be seen as
representations of a crossed product algebra or as usual Kohn-Nirenberg
operators defined in $\mathbb R^{2N}$. The symbols of these Kohn-Nirenberg operators are computed from the magnetic Weyl symbol
by applying the magnetic modulation mapping followed by a partial Fourier transform.

{\bf Acknowledgements:} M. M\u antoiu is partially supported by {\it N\'ucleo Cientifico ICM P07-027-F "Mathematical Theory of
Quantum and Classical Magnetic Systems"} and by Chilean Science Foundation {\it Fondecyt} under the Grant 1085162.
His interest in modulation spaces was raised by a very enjoyable visit at the Universiy of
Vienna in February, 2009. He would like to thank the members of the NuHAH group and especially Professor Hans Feichtinger for
hospitality and many useful discussions.

R. Purice acknowledges partial support from the Contract no. 2-CEx 06-11-18/2006.

%...............................................................................................
\section{Recall of the magnetic Weyl calculus}
%.................................................................................................

In this Section we recall the structure of the observable algebras of a particle in a variable magnetic field.
We follow the references \cite{MP1}, \cite{MP3} and \cite{IMP1}, which contain further details
and technical developments.

The physical system we consider consists in a spin-less particle moving in the euclidean space $\X:=\mathbb R^N$ under
the influence of a magnetic field. We denote by $\X^*$ the dual space of $\X$. The duality is given simply by
$\X\times\X^*\ni(x,\xi)\mapsto x\cdot\xi$. The phase space $\Xi:=T^*\X\equiv\X\times\X^*$, containing points
$X=(x,\xi)$, $Y=(y,\eta)$, $Z=(z,\zeta)$, is endowed with the standard symplectic form
\begin{equation*}\label{simp}
\sigma(X,Y)\equiv\sigma[(x,\xi),(y,\eta)]:=y\cdot \xi-x\cdot\eta.
\end{equation*}

The magnetic field is a continuous closed $2$-form $B$ on $\X$ ($dB=0$), given by matrix-component functions
$B_{jk}=-B_{kj}:\X\rightarrow\mathbb R,\ j,k=1,\dots,N.$
It defines quantum observable composition in terms of its fluxes through triangles.
If $a,b,c\in \X$, then we denote by $<a,b,c>$ the triangle in $\X$ of vertices $a,b$ and $c$ and set
$$
\Gamma^B(<a,b,c>):=\int_{<a,b,c>}B
$$
for the flux of $B$ through it (invariant integration of a $2$-form through a $2$-simplex). Then the formula
\begin{equation}\label{composition}
\left(f\#^B g\right)(X):=\pi^{-2N}\int_\Xi dY\int_\Xi dZ\,\exp\left[-2i\sigma(X-Y,X-Z)\right] \times
\end{equation}
$$
\times\exp\left[-i\Gamma^B(<x-y+z,y-z+x,z-x+y>)\right]f(Y)g(Z)
$$
defines a formal associative composition law on functions $f,g:\Xi\rightarrow \mathbb C$.

The formula (\ref{composition}) makes sense and have nice properties under various circumstances. For example, if the
components $B_{jk}$ belong to $C^\infty_{\rm{pol}}(\X)$, the class of smooth functions on $\X$ with polynomial bounds on all the
derivatives, then the Schwartz space $\S(\Xi)$ is stable under $\#^B$. The dual of $\S(\Xi)$ being denoted by
$\S^*(\Xi)$ (tempered distributions), one also has
$$
\#^B:\S(\Xi)\times\S^*(\Xi)\rightarrow\S^*(\Xi)\ \ \rm{and}\ \  \#^B:\S^*(\Xi) \times\S(\Xi)\rightarrow\S^*(\Xi).
$$
Denoting by $\mathcal M^B(\Xi)$ the largest subspace of $\S^*(\Xi)$ for which
$$
\#^B:\S(\Xi)\times\mathcal M^B(\Xi)\rightarrow\S(\Xi)\ \  \rm{and}\ \ \#^B:\mathcal
M^B(\Xi)\times\S(\Xi)\rightarrow\S(\Xi),
$$
it can be shown that $\mathcal M^B(\Xi)$ is an involutive algebra under
$\#^B$ and under complex conjugation, for which
$$
\#^B:\S^*(\Xi)\times\mathcal M^B(\Xi)\rightarrow\S^*(\Xi)\ \ \rm{and}\ \  \#^B:\mathcal M^B(\Xi)
\times\S^*(\Xi)\rightarrow\S^*(\Xi).
$$
This is quite a large class of distributions, containing all the bounded measures as well as the class
$C^\infty_{\rm{pol,u}}(\Xi)$ of all smooth functions for which all the derivatives are bounded by some polynomial
(depending on the function, but not on the order of the derivative). In addition, if we assume that all the
derivatives of the functions $B_{jk}$ are bounded, the H\"ormander classes of symbols $S^m_{\rho,\delta}(\Xi)$
composes in the usual way under $\#^B$.

Being a closed $2$-form in $\X=\mathbb R^N$, the magnetic field is exact: it can be written as $B=dA$ for some continuous
$1$-form $A$ (called {\it vector potential}). Vector potentials enter by their circulations
$$
\Gamma^A([x,y]):=\int_{[x,y]}A
$$
through segments $[x,y]:=\{tx+(1-t)y\mid t\in[0,1]\}$.
For a vector potential $A$ with $dA=B$, let us define
\begin{equation}\label{op}
\left[\Op(f)u\right](x):=(2\pi)^{-N}\int_\X\int_{\X^*}dy\,d\xi\,\exp\left[i(x-y)\cdot\xi\right]
\exp\left[-i\Gamma^A([x,y])\right]f\left(\frac{x+y}{2},\xi\right)u(y).
\end{equation}
For $A=0$ one recognizes the Weyl quantization, associating to functions or distributions on $\Xi$ linear
operators acting on function spaces on $\X$.

The space $L^2(\Xi)$ is a $^*$-algebra under $\#^B$ and complex conjugation and $\Op$
is an isomorphism of $L^2(\Xi)$ on the Hilbert space $\mathbb B_2(\H)$ of all the Hilbert-Schmidt operators on $\H=L^2(\X)$.

Suitably interpreted (by using duality arguments), $\Op$ defines a representation of the $^*$-algebra
$\mathcal M^B(\Xi)$ by linear continuous operators $:\mathcal S(\X)\rightarrow\mathcal S(\X)$, i.e.
$$
\Op(f\#^B g)=\Op(f)\Op(g)\ \  {\rm and}\ \  \Op(\overline f)=\Op(f)^*
$$
for any $f,g\in\mathcal M^B(\Xi)$. In addition, $\Op$ restricts to an isomorphism from $\mathcal S(\Xi)$ to $\mathbb
B[\mathcal S^*(\X),\mathcal S(\X)]$ and extends to an isomorphism from $\mathcal S^*(\Xi)$ to $\mathbb B[\mathcal
S(\X),\mathcal S^*(\X)]$ (we set $\mathbb B(\mathcal R,\mathcal T)$ for the family of all linear continuous operators
between the topological vector spaces $\mathcal R$ and $\mathcal T$).

An important property of (\ref{op}) is {\it gauge covariance}: if $A'=A+d\rho$  defines the same magnetic field as $A$, then
$\mathfrak Op^{A'}(f)=e^{i\rho}\Op(f)e^{-i\rho}$. Such a property would not hold for the wrong quantization,
appearing in the literature
\begin{equation*}\label{bufnita}
\left[\mathcal{O}p_{A}(f)u\right](x):=(2\pi)^{-N}\int_\X\int_{\X^*}dy\,d\xi\,
\exp\left[i(x-y)\cdot\xi\right]f\left(\frac{x+y}{2},\xi-A\left(\frac{x+y}{2}\right)\right)u(y).
\end{equation*}
To justify (\ref{op}) we define a family $(\e_X)_{X\in\Xi}$ of functions that will play an important role in the
sequel
\begin{equation}
\e_X(Z):=\exp\{-i\sigma(X,Z)\},\ \ Z\in\Xi.
\end{equation}
Actually they are elements of $C^\infty_{\rm{pol,u}}(\Xi)\subset\mathcal M^B(\Xi)$. One checks easily that
\begin{equation}\label{fac}
\e _X\,\#^B\e_Y= \Omega^B(X,Y)\,\#^B\,\e_{X+Y},
\end{equation}
where $\Omega^B:\Xi\times\Xi\rightarrow C(\mathcal{X},U(1))$ is the $2$-cocycle defined by the canonical
symplectic form and by the magnetic field $B$:
\begin{equation}\label{caf}
\Omega^B(X,Y)(z)\equiv\Omega^B(X,Y;z):=\exp\left[\frac{i}{2}\,\sigma(X,Y)\right]\omega^B(X,Y;z),
\end{equation}
with
\begin{equation}\label{fca}
\omega^B(X,Y;z):=\exp\left[-i\Gamma^B(<z,z+x,z+x+y>)\right].
\end{equation}
Suitable functions $f:\Xi\rightarrow\mathbb C$ can be expressed as
$$
f(Y)=(2\pi)^{-N}\int_\Xi dX\,(\mathfrak F f)(X)e^{-i\sigma(X,Y)}=(2\pi)^{-N}\int_\Xi dX\,(\mathfrak F f)(X)\e_X(Y),
$$
where $\mathfrak F f$ is {\it the symplectic Fourier transform of} $f$, so a good quantization should have the property
\begin{equation}\label{justi}
\Op(f)=(2\pi)^{-N}\int_\Xi dX\,(\mathfrak F f)(X)\,\Op(\e_X).
\end{equation}
Thus, the problem is to justify a choice for the operators $\mathfrak{op}^A(X):=\Op(\e_X)$ acting in $\H=L^2(\X)$.

In the presence of a magnetic field $B=dA$, a basic family of self-adjoint operators is $(Q_1,\dots,Q_N;\Pi^A_{1},
\dots,\Pi^A_{N})$, where $Q_j$ is the operator of multiplication by the coordinate function $x_j$ and
$\Pi^A_{j}:=-i\partial_j-A_j$ is the $j'$th component of the magnetic momentum. They satisfy the commutation relations
\begin{equation*}\label{comut}
i[Q_j,Q_k]=0,\ \ \ i[\Pi^A_{j},Q_k]=\delta_{jk},\ \ \ i[\Pi^A_{j},\Pi^A_{j}]=B_{jk}.
\end{equation*}
One gets (\ref{op}) as a consequence of (\ref{justi}) admitting that the quantization of the function
$X\mapsto\e_Y(X)$ should be the unitary operator
\begin{equation}\label{now}
\mathfrak{op}^A(y,\eta):=\mathfrak{Op}^A\left(\e_{(y,\eta)}\right)
=\exp\left[-i\sigma((y,\eta),(Q,\Pi^A))\right]=\exp\left[-i(Q\cdot\eta-y\cdot\Pi^A)\right],
\end{equation}
given by the explicit formula
$$
[\mathfrak{op}^A(y,\eta)u](x)=e^{-i(x+\frac{y}{2})\cdot\eta}e^{-i\Gamma^A\,([x,x+ y])}u(x+y),\ \ \ x,y\in\X,\
\eta\in\X^*,\ u\in\mathcal H.
$$
As a consequence of (\ref{fac}), one has
\begin{equation}\label{asa}
\mathfrak{op}^A(X)\,\mathfrak{op}^A(Y)=\Omega^B(X,Y;Q)\,\mathfrak{op}^A(X+Y),\ \ \ \forall X,Y\in\Xi,
\end{equation}
where $\Omega^B(X,Y;Q)$ is the operator of multiplication by the function $z\mapsto [\Omega^B(X,Y)](z)\equiv\Omega^B
(X,Y;z)$ given at (\ref{caf}) and (\ref{fca}).

The operator norm $\parallel\cdot\parallel$ on $\mathbb B\left(\mathcal H\right)$ being relevant to Quantum Mechanics, we
pull it back to symbols (the nicer, completely intrinsic approach can be found in \cite{MPR1}, \cite{MP3}).
So let us set
$$
\parallel\cdot\parallel^B:\mathcal{S}(\Xi)\rightarrow\mathbb R_+, \ \
\parallel f \parallel^B=\parallel \Op(f) \parallel.
$$

By gauge covariance, it is clear that $\parallel\cdot\parallel^B$ only depends on the magnetic field
$B$ and not on the vector potential $A$. We denote by $\mathfrak A^B(\Xi)$ the completion of
$\mathcal{S}(\Xi)$ under $\parallel\cdot\parallel^B$. It is a $C^*$-algebra that can be identified to a vector subspace of
$\mathcal{S}^*(\Xi)$ and $\Op:\mathfrak{A}^B(\Xi)\rightarrow \mathbb B(\mathcal H)$ is a faithful $^*$-representation, with
$\Op\left[\mathfrak A^B(\Xi)\right] = \mathbb K\left(\mathcal H\right)$, the $C^*$-algebra of compact operators in $\mathcal H$.

Many other useful $C^*$-algebras can be defined in this manner. An important one is $\mathfrak C^{B}(\Xi)$, defined such
that $\Op:\mathfrak C^B(\Xi)\rightarrow \mathbb B(\mathcal H)$ be an isomorphism. The "magnetic version" of the
Calderon-Vaillancourt theorem, proved in \cite{IMP1}, says that if $B_{jk}\in BC^\infty (\X),\ \ j,k = 1,\dots,N$, then
the Fr\'echet space $BC^\infty (\Xi)$ of smooth functions on $\Xi$ having bounded derivatives of any order is
continuously embedded in $\mathfrak C^B(\Xi)$. We note that $L^2(\Xi)$ and $\mathfrak A^B(\Xi)$ are $^*$-ideals
in $\mathfrak C^B(\Xi)$.

We set
\begin{equation}\label{tracas}
\Theta^B_Z(f):=\e_{-Z}\,\#^B f\,\#^B\e_{Z}
\end{equation}
for the family of {\it magnetic translations in phase-space}, introduced in \cite{IMP'} and used for a Beals-type
characterization of magnetic pseudodifferential operators by commutators.
They are automorphisms of the $^*$-algebras $L^2(\Xi),\mathfrak A^B(\Xi),\mathfrak C^B(\Xi),\mathfrak M^B(\Xi)$
and reduce, for $B=0$, to the usual translations $\left[\Theta_Z(f)\right](X):=f\left(X+Z\right)$.

We shall need an explicit form of $\Theta^B_Z$, obtained in \cite{IMP'}.
For this we define
 the following commutative {\it mixed product} (a mixture between point-wise multiplication in the first variable
 and convolution in the second):
\begin{equation}\label{star}
 (F\star g)(x,\xi)\,:=\,\int_{\mathcal{X}^\prime}d\eta\,F(x,\xi-\eta)\,g(x,\eta).
\end{equation}

\begin{proposition}\label{magn-transl}
For any 3 points $x,y,z\in \mathcal{X}$ let us define the parallelogram
\begin{equation*}
 \mathcal{P}(x;y,z)\,:=\,\{x+sy+tz\mid s\in[-1/2,1/2],\,t\in[-1,0]\},
\end{equation*}
having edges parallel to the vectors $y$ and $z$, respectively. We consider the distribution
\begin{equation}
 \Omega^B[\mathcal{P}(x;y,z)]=\exp\left\{-i\Gamma^B[\mathcal{P}(x;y,z)]\right\}=
 \exp\left\{-i\sum_{j,k=1}^Ny_jz_k\int_{-1/2}^{1/2}ds\int_{-1}^0dtB_{jk}(x+sy+tz)\right\}
\end{equation}
and its Fourier transform with respect to the second variable:
\begin{equation}
 \widetilde{\Omega^B_\mathcal{P}}[z](x,\xi)\,:=(2\pi)^{-N}\,\int_{\mathcal{X}}dy\, e^{-iy\cdot\xi}\Omega^B[\mathcal{P}
 (x;y,z)].
\end{equation}
For $Z=(z,\zeta)\in\Xi$ and $f\in L^1(\Xi)$ we have
\begin{equation}\label{aut}
\Theta^B_Z(f)\,=\,\widetilde{\Omega^B_\mathcal{P}}[z]\star\Theta_Z[f].
\end{equation}
\end{proposition}

More generally, we set
\begin{equation}\label{tracass}
\Theta^B_{Y,Z}(f):=\e_{-Y}\,\#^B f\,\#^B\e_{Y-Z},
\end{equation}
which makes sense for every $f\in\mathcal M^B(\Xi)$. We have $\Theta^B_{Y,0}=\Theta^B_{Y}$.

%...........................................................................................................
\section{Modulation mappings}\label{sashaa}
%...........................................................................................................

We only assume that $B$ has continuous components.  On functions defined on $\Xi$ or on $\Xi\times\Xi$, respectively, we
will use the "real" scalar products
$$
\<f,g\>:=\int_\Xi dXf(X)g(X),\ \ \ \ \ \<\!\<F,G\>\!\>:=\int_\Xi\int_\Xi dXdYF(X,Y)G(X,Y).
$$

\begin{definition}\label{justify}
{\it The magnetic modulation mapping $M^B_{h}:L^2(\Xi)\rightarrow L^2(\Xi\times\Xi)$ associated to $h\in L^2(\Xi)$} is
\begin{equation}\label{begin}
\left[M^B_{h}(f)\right](X,Y):=\<\e_{-X}\#^B f\#^B\e_{X-Y},h\>=\langle\Theta^B_{X,Y}[f],h\rangle.
\end{equation}

\end{definition}

\noindent
{\bf Remark.} By using the formula (\ref{croco}), one can also write
$$
\left[M^B_{h}(f)\right](X,Y)=\<h\#^B\e_{-X}\#^B f,\e_{X-Y}\>=\[\mathfrak F\(h\#^B\e_{-X}\#^B f\)\](X-Y).
$$
In the case $B=0$ one gets
\begin{equation}\label{amea}
\[M_{h}(f)\](X,Y)=\<\e_{-X}\#f\#\e_{X-Y},h\>=\exp\left[\frac{i}{2}\,\sigma(X,Y)\right]\<\Theta_X(f)\#\e_{-Y},h\>,
\end{equation}
which is different from the standard choice (the Short Time Fourier Transform)
\begin{equation*}\label{alor}
\[\mathcal V_{h}(f)\](X,Y)=\<\e_{-Y}\Theta_X(f),h\>.
\end{equation*}
The main difference is the replacement of point-wise multiplication by the Weyl product, and this might be an advantage
for studying the algebraic properties of the modulation mapping.

\medskip
The next {\it orthogonal relations} justify formally Definition \ref{justify}

\begin{theorem}\label{incep}
For $f,g,h,k\in L^2(\Xi)$, one has
\begin{equation}\label{ortog}
\<\!\<\overline{M^B_{h}(f)},M^B_{k}(g)\>\!\>=\<\overline h,k\>\<\overline f,g\>.
\end{equation}
In particular, $\frac{1}{\parallel h\parallel}M^B_{h}$ is an isometry.
\end{theorem}

The proof will use several properties which are gathered in the next Lemma:

\begin{lemma}\label{basica}
(a) For any $f_1,f_2,f_3\in L^2(\Xi)$
\begin{equation}\label{croco}
\<f_1\#^B f_2,f_3\>=\<f_1,f_2\#^B f_3\>=\<f_2,f_3\#^B f_1\>.
\end{equation}
The same is true if one of the three functions is replaced by $\mathfrak e_X$, for some $X\in\Xi$.

\medskip
(b) One has in weak sense
\begin{equation}\label{pleasca}
\int_\Xi dZ\,|\e_Z\rangle\langle\e_{-Z}|=1.
\end{equation}

(c) For any $f,g\in L^1(\Xi)$
\begin{equation}\label{ceoare}
\int_\Xi\int_\Xi dYdZ\left[\Theta^B_Z(f)\right](Y)\,g(Y)=\int_\Xi dZ f(Z)\int_\Xi dY g(Y).
\end{equation}
\end{lemma}

\begin{proof}
(a) We proved in \cite{MP1} the case $f_1,f_2,f_3\in\S(\Xi)$. For the other cases one rewrites the proof more carefully or
uses an approximation argument.

\medskip
(b) One has for $f,g\in L^2(\Xi)$
$$
\int_\Xi dZ\,\langle f,\e_Z\rangle\langle\e_{-Z},g\rangle=
(2\pi)^{2N}\int_\Xi dZ\,(\mathfrak F f)(Z)\,(\mathfrak F g)(-Z)=
$$
$$
=(2\pi)^{2N}\int_\Xi dZ\,(\mathfrak F f)(Z)\,\overline{(\mathfrak F \overline g)(Z)}=
\int_\Xi dZ\,f(Z)\,g(Z)=\langle f,g\rangle.
$$

(c) We apply Proposition \ref{magn-transl} and use the Fourier inversion formula:
$$
\int_\Xi\int_\Xi dYdZ\left[\Theta^B_Z(f)\right](Y)\,g(Y)=
\int_\Xi\int_\Xi dYdZ\left[\left(\widetilde\Omega^B_\mathcal P[z]\right)\star \Theta_Z(f)\right](Y)\,g(Y)=
$$
$$
=\int_\Xi\int_\Xi dYdZ\int_{\X^*}d\xi\left(\widetilde\Omega^B_\mathcal P[z]\right)(y,\xi)\,
\left[\Theta_Z(f)\right](y,\eta-\xi)\,g(Y)=
$$
$$
=(2\pi)^{-N}\int_{\X}dy\int_{\X}dz\int_{\X^*}d\eta\int_{\X^*}d\zeta\int_{\X}dx\int_{\X^*}d\xi\,e^{-ix\cdot\xi}\,
\Omega^B\left[\mathcal P(y;x,z)\right]\,f(y-z,\eta-\xi-\zeta)\,g(y,\eta)=
$$
$$
=(2\pi)^{-N}\int_{\X}dy\int_{\X}dz\int_{\X^*}d\eta\int_{\X^*}d\zeta\int_{\X}dx\int_{\X^*}d\nu\,e^{ix\cdot\zeta}
e^{ix\cdot(\nu-\eta)}\,\Omega^B\left[\mathcal P(y;x,z)\right]\,f(y-z,\nu)\,g(y,\eta)=
$$
$$
=\int_{\X}dy\int_{\X}dz\int_{\X^*}d\eta\int_{\X^*}d\nu\,\Omega^B\left[\mathcal P(y;0,z)\right]
\,f(y-z,\nu)\,g(y,\eta)=\int_\Xi dZ f(Z)\int_\Xi dY g(Y).
$$
\end{proof}

\begin{proof}{\it of the Theorem.}
Using Lemma \ref{basica} we compute
$$
\<\!\<\overline{M^B_{h}(f)},M^B_{k}(g)\>\!\>=\int_\Xi\int_\Xi dXdY\<\overline{\e_{-X}\#^Bf\#^B\e_{X-Y}},\overline h\>
\<\e_{-X}\#^Bg\#^B\e_{X-Y},k\>=
$$
$$
=\int_\Xi\int_\Xi dXdY\<\e_{Y-X}\#^B\overline{f}\#^B\e_{X},\overline h\>\<\e_{-X}\#^Bg\#^B\e_{X-Y},k\>=
$$
$$
=\int_\Xi\int_\Xi dXdY \< \overline{f}\#^B\e_{X}\#^B\overline h,\e_{Y-X}\>\<\e_{X-Y},k\#^B\e_{-X}\#^Bg\>=
$$
$$
=\int_\Xi dX\<\overline{f}\#^B\e_X\#^B \overline h,k\#^B\e_{-X}\#^B g\>=
\int_\Xi dX\<\e_X\#^B \overline h\#^B k,\e_{-X}\#^B g\#^B\overline f\>=
$$
$$
=\int_\Xi dX\<\Theta^B_X(\overline h\#^B k),g\#^B\overline f\>=\int_\Xi dX\,(\overline h\#^B k)(X)
\int_\Xi dY\,(g\#^B\overline f)(Y)=\<\overline h,k\>\<\overline f,g\>.
$$
\end{proof}

\begin{corollary}\label{inver}
We have {\rm the inversion formula}:
\begin{equation}\label{invo}
\(M^B_k\)^*M^B_h=\<h,\overline k\>{\rm id}.
\end{equation}
\end{corollary}

The adjoint $\left(M^B_{h}\right)^*:L^2(\Xi\times\Xi)\rightarrow L^2(\Xi)$ is given explicitly by
\begin{equation}\label{adjunctu}
\left(M^B_{h}\right)^*(G):=\int_\Xi \int_\Xi dXdY\,G(X,Y)\,\e_X\#^B {h}\#^B \e_{Y-X}
=\int_\Xi \int_\Xi dXdY\,G(X,Y)\left(\Theta^B_{X,Y}\right)^{-1}(h),
\end{equation}

\noindent
{\bf Remark.}
The Theorem  suggests defining
\begin{equation}\label{ldoiul}
M^B:L^2(\Xi\times\Xi)\cong L^2(\Xi)\otimes L^2(\Xi)\rightarrow L^2(\Xi\times\Xi),\ \ \ \ \ M^B(f\otimes h):=M^B_h(f).
\end{equation}
One has $M^B_h=M^B\circ J_h$, where for any $h\in L^2(\Xi)$ we set
\begin{equation}\label{margay}
J_h:L^2(\Xi)\rightarrow L^2(\Xi\times\Xi),\ \ \ \ \ J_h(f):=f\otimes h.
\end{equation}
The adjoint is given by
\begin{equation}\label{marguy}
J_h^*:L^2(\Xi\times\Xi)\rightarrow L^2(\Xi),\ \ \ \ \ \[J_h^*(F)\](X):=\<F(X,\cdot),h(\cdot)\>,
\end{equation}
and it satisfies
\begin{equation}\label{ci}
J_k^*J_h=<h,k>{\rm id},\ \ \ \ \ J_hJ_k^*=1\otimes\left(|h><k|\right)=1\otimes{\rm Int}_{h\otimes k}.
\end{equation}
While $M^B$ is an isomorphism, $\parallel h\parallel^{-1}J_h$ is only an isometry with range $L^2(\Xi)\otimes\{h\}$.

\medskip
We turn now to the algebraic properties of the magnetic modulation mapping.
On functions $:\Xi\times\Xi\rightarrow\mathbb C$ we use {\it the crossed product composition}
\begin{equation}\label{crosu}
(F\diamond G)(X,Y):=\int_\Xi dZ\,F(X,Z)\,G(X-Z,Y-Z)
\end{equation}
and the involution $F^*(X,Y):=\overline{F(X-Y,-Y)}$. The "crossed product" feature can be seen if we write (\ref{crosu}) as
$$
(F\diamond G)(Y):=\int_\Xi dZ\,F(Z)\,\Theta_Z\left[G(Y-Z)\right].
$$
This is an equality between functions defined on $\Xi$, so it must be evaluated on $X\in\Xi$ by using notations as
$[F(Z)](X):=F(X,Z)$. The action of $\Xi$ on itself given by $\Theta_Z(X):=X+Z$ is transferred to functions by
$\Theta_Z(g):=g\circ\Theta_{-Z}$. For various function spaces on $\Xi$ one can define $C^*$-dynamical systems and they
generate naturally crossed product $C^*$-algebras of functions or distributions defined on $\Xi\times\Xi$. We refer to
\cite{W} for general information on this topic; we are going to study the connection of crossed products with magnetic
modulation spaces in a further publication.

For the moment we only notice that $L^2(\Xi\times\Xi)$ is a $^*$-algebra with the structure indicated above. To see this,
one might recall the kernel multiplication
$$
(K\tilde\diamond L)(X,Y):=\int_\Xi dZ\,K(X,Z)\,L(Z,Y)
$$
and the involution $K^{\tilde *}(X,Y):=\overline{K(Y,X)}$ and perform the change of variables $(X,Y)\mapsto(X,X-Y)$.

\begin{theorem}\label{fundamentala}
If $h\#^B h=h=\overline h\ne 0$, then $M^B_{h}:L^2(\Xi)\rightarrow L^2(\Xi\times\Xi)$ is an injective morphism of $^*$-algebras.
\end{theorem}

\begin{proof}
We are going to use (\ref{croco}) and (\ref{pleasca}) to show that
\begin{equation}\label{minunata}
M^B_{h}(f)\diamond M^B_{k}(g)=M^B_{k\,\#^B h}\(f\#^B g\),
\end{equation}
and then take $h=k$. One has
$$
\[M^B_{h}(f)\diamond M^B_{k}(g)\](X,Y)=\int_\Xi dZ\[M^B_{h}(f)\](X,Z)\[M^B_{k}(g)\](X-Z,Y-Z)=
$$
$$
=\int_\Xi dZ\<\e_{-X}\#^B f\#^B\e_{X-Z},h\>\<\e_{Z-X}\#^B g\#^B\e_{X-Y},k\>=
$$
$$
=\int_\Xi dZ\<h\#^B\e_{-X}\#^B f,\e_{X-Z}\>\<\e_{Z-X},g\#^B\e_{X-Y}\#^B k\>=
$$
$$
=\<h\#^B\e_{-X}\#^B f,g\#^B\e_{X-Y}\#^B k\>=
$$
$$
=\<\e_{-X}\#^B(f\#^B g)\#^B\e_{X-Y},k\#^B h\>=\[M^B_{k\,\#^B h}(f\#^B g)\](X,Y).
$$
For the involution:
$$
\[M^B_{h^*}(f^*)\](X,Y)=\<\e_{-X}\#^B\overline{f}\#^B\e_{X-Y},\overline{h}\>=
$$
$$
\overline{\<\e_{Y-X}\#^Bf\#^B\e_X,h\>}=\overline{\[M^B_{h}(f)\](X-Y,-Y)}=\[M^B_{h}(f)\]^*(X,Y).
$$
The injectivity follows from Theorem \ref{incep}.
\end{proof}

\noindent
{\bf Remark.}
One could also use the composition law
\begin{equation}\label{dublu}
\square^B:L^2(\Xi\times\Xi)\times L^2(\Xi\times\Xi)\rightarrow L^2(\Xi\times\Xi),\ \ \ \ \
(f\otimes h)\square^B(g\otimes k):=(f\#^B g)\otimes (k\#^B h)
\end{equation}
and the usual involution on $L^2(\Xi\times\Xi)$ given by complex conjugation.
Then
\begin{equation}\label{udatu}
M^B:\(L^2(\Xi\times\Xi),\square^B\)\rightarrow\(L^2(\Xi\times\Xi),\diamond\)
\end{equation}
is an isomorphism of $^*$-algebras.

%............................................................................................................
\section{Connections with the Bargmann transform}\label{finalaa}
%............................................................................................................

The magnetic analog of the Bagmann transform requires a suitable family of coherent states.
The main idea for introducing them will be to use the magnetic Weyl system (\ref{now})
to propagate a given state, corresponding to $X=0$, to a family of states indexed by the points $X$ of phase space.
In order to insure gauge-covariance, this state must have a good a priori dependence of the vector potential $A$.
Of course this seems to fit the group-theoretical strategy to generate coherent states, but we stress that (\ref{asa})
collapses to the definition of a projective representation only in the very simple case of a constant magnetic field.

We note that with a proper implementation of Planck's constant $\hb$, one proves (\cite{MP2}) convergence of the quantum
algebra of observables to the classical one in the sense of strict deformation quantization (cf. \cite{Ri1,La3}).
We intend to study in a future publication the dependence of coherent states and the associated
Berezin-Toeplitz operators on the Planck constant, in the framework of deformation quantization (cf. \cite{La1,La2,La3}).

Let us fix a unit vector, $v\in\mathcal H := L^2 (\X)$. For any
choice of a continuous vector potential $A$ generating the magnetic field $B$ and for any $Y\in\Xi$, we define
{\it the family of magnetic coherent vectors}
$$
v^A\equiv v^A(0):=e^{i\Gamma^A([0,Q])}v,\ \ \ \ \ v^A(Y):=\mathfrak{op}^A(-Y)v^A.
$$
Explicitly
\begin{equation}\label{mur}
 \left[v^A(Y)\right](x)=e^{i(x-\frac{y}{2})\cdot\eta}e^{-i\Gamma^A([x,x-y])}e^{i\Gamma^A([0,x-y])}v(x-y).
\end{equation}
Note that for the standard Gaussian $v(x)=\pi^{-N/4}e^{-x^2/2}$ and
for $A=0$, one gets the usual coherent states of Quantum Mechanics (see \cite{AAGM}).
It is easy to show that in weak sense
\begin{equation}\label{zic}
\int_{\Xi}\frac{dY}{(2\pi)^N}\,|v^A(Y)\rangle\langle v^A(Y)|=1.
\end{equation}
We are not going to prove this simple result, since it will not be needed in the sequel.

Since the pure state space of $\mathbb K(\H)$ can be identified with $\mathbb P(\H)$ (the family of all
self-adjoint one-dimensional projections
in $\H$) and considering the isomorphism $\mathfrak{Op}^A:\mathfrak A^B(\Xi)\rightarrow \mathbb K(\H)$,
we introduce families of coherent states on the two $C^*$-algebras:

\begin{definition}\label{enfin}
For any $Z\in\Xi$ we define
$$
\mathfrak v^A(Z):\mathbb K(\H)\rightarrow\C,\ \ \ \ \ \mathfrak v^B(Z):\mathfrak A^B(\Xi)\rightarrow\C
$$
by
$$
\left[\mathfrak v^A(Z)\right](S):=
{\rm Tr}\left(\left\vert v^A(Z)\right> \left< v^A(Z) \right\vert S\right)=
\left< v^A(Z),S\,v^A(Z)\right>,\ \ \ \ \ \forall S\in\mathbb K(\H)
$$
and
$$
\left[\mathfrak v^B(Z)\right](f):=\left[\mathfrak v^A(Z)\right]\left[\mathfrak{Op}^A(f)\right]=
\left< v^A(Z),\mathfrak{Op}^A(f)\,v^A(Z)\right>,\ \ \ \ \ \forall f\in\mathfrak A^B(\Xi).
$$
\end{definition}

One can write
$$
\left[\mathfrak v^B(Z)\right](f)=\left< v^A,\mathfrak{op}^A\left(Z\right)\Op(f)
\mathfrak{op}^A\left(-Z\right)v^A\right>=\left< v^A,
\Op\left(\e_{Z}\,\#^B f\,\#^B\e_{-Z}\right)v^A\right>.
$$
Using (\ref{tracas}) and setting $\mathfrak v^B:=\mathfrak v^B(0)$, one has
$\mathfrak v^B(Z)=\mathfrak v^B\circ\Theta^B_Z$.
The intrinsic notation $\mathfrak v^B(Z)$ is justified by a straightforward computation based on Stokes' Theorem, leading to

\begin{equation}\label{insfarsit}
\left[\mathfrak v^B(Z)\right](f)=(2\pi)^{-N}\int_\X\int_\X\int_{\X^*}dx\,dy\,d\xi\,
e^{i(x-y)\cdot(\xi-\zeta)}f\left(\frac{x+y}{2},\xi\right)\cdot
\end{equation}
$$
\cdot\exp\left\{i\left[\Gamma^B(<y,x,x-z>)+\Gamma^B(<y,x-z,0>)+\Gamma^B(<y,0,y-z>)
\right]\right\}\overline{v(x-z)}v(y-z).
$$

A convenient setting is obtained after making a unitary transformation, generalizing the classical Bargmann
transformation; the associated Bargmann-type space is a Hilbert
space with reproducing kernel.

\begin{definition}
(a) The mapping $\mathcal U^A_v : L^2(\X)\rightarrow L^2 \left(\Xi;\frac{dX}{(2\pi)^N}\right)$,
\begin{equation}\label{barr}
\left(\mathcal U^A_v u\right)(X):=\left< v^A (X), u\right>=\left<v,\mathfrak{op}^A(X)u\right>
\end{equation}
is called {\rm the Bargmann transformation} corresponding to the family of coherent states $(v^A (X))_{X\in\Xi}$.

(b) The subspace $\mathcal K^A_v :=\mathcal U^A_v L^2 (\X)\subset L^2(\Xi)$
is called {\rm the Bargmann space} corresponding to the family of coherent states $(v^A(X))_{X\in \Xi}$.
\end{definition}

The proofs of the statements bellow are straightforward and not specific to our magnetic framework (see \cite{La3}, section
II.1.5 for instance):

\begin{proposition}\label{nespec}
(a) $\mathcal U^A_v$ is an isometry with adjoint
$$
\left(\mathcal U^A_v\right)^* : L^2(\Xi)\rightarrow L^2(\X),\ \ \left(\mathcal U^A_v\right)^* \Phi
:=\int_{\Xi}\frac{dX}{(2\pi)^N}\Phi (X) v^A(X)
$$
and final projection $P^A_v:=\mathcal U^A_v\left(\mathcal U^A_v\right)^*\in\mathbb P [L^2(\Xi)]$, with $P^A_v
L^2(\Xi)=\mathcal K^A_v$.

(b) The kernel of this projection
$$
K^A_v:\Xi\times\Xi\rightarrow\mathbb C,\ \ K^A_v(X,Y):=\left< v^A(X),  v^A(Y)\right>
$$
is a continuous function and it is a reproducing kernel for
$\mathcal K^A_v$:
$$
\Phi (X)=\int_{\Xi}\frac{dX}{(2\pi)^N} K^A_v(X,Y)\Phi (Y),\
\ \forall\,X\in\Xi,\ \ \forall\,\Phi\in\mathcal K^A_v.
$$
(c) The Bargmann space is composed of continuous functions and contains all the vector $K^A_v(X,\cdot),\,X\in\Xi$. The
evaluation maps $\mathcal K^A_v\ni\Phi\rightarrow\Phi(X)\in\mathbb C$ are all continuous.
\end{proposition}

For various types of vectors $u,v:\Xi\rightarrow \mathbb C$ we define {\it the magnetic Wigner transform} $V^A_{u,v}$ by
\begin{equation}\label{ti}
\langle u,\Op(f)v\rangle=:\int_\Xi dX\,f(X)V^A_{u,v}(X).
\end{equation}
One gets easily
$V^A_{v,u}=(2\pi)^{N/2}\,[(1\otimes\mathcal F)\circ C]\,[\overline\gamma^A\cdot(u\otimes\overline{v})]$,
with $\gamma^A(a,b):=e^{-i\Gamma^A([a,b])}$, which may also be written
\begin{equation}\label{wigner}
V^A_{v,u}(z,\zeta)=\int_\X dy\,e^{i\,y\cdot\zeta}\,\overline\gamma^A
\left(z+\frac{y}{2},\,z-\frac{y}{2}\right)u \left(z+\frac{y}{2}\right)\overline{v}\left(z-\frac{y}{2}\right).
\end{equation}
It follows that $u,v\in L^2(\X)\ \Longrightarrow\ V^A_{u,v}\in L^2(\Xi)$.
By a direct computation one gets $|u><v|=\Op\(V^A_{u,v}\)$, which describes all the rank-one operators in $\H$.
This suggests studying the modulation mapping $M^B_{h}$ in the particular case in which
$h$ is the Wigner transform $h(B,v):=V^A_{v^A,v^A}$, given explicitly by
\begin{equation}\label{formol}
\[h(B,v)\](z,\zeta)=\int_\X dy\,e^{i\,y\cdot\zeta}\,\exp\[i\Gamma^B
\left(\<0,z+\frac{y}{2},\,z-\frac{y}{2}\>\right)\]v\left(z+\frac{y}{2}\right)\overline{v\left(z-\frac{y}{2}\right)}.
\end{equation}

Let us also set $\mathbb U^A_v[T]:=\mathcal U^A_vT\left(\mathcal U^A_v\right)^*$. We denote by
$REP$ the Schr\"odinger representation of $L^2(\Xi\times\Xi)$ in $L^2(\Xi)$ given by
$$
[REP(F)\Phi](X):=\int_\Xi dY\,F(X,X-Y)\Phi(Y).
$$

\begin{proposition}\label{ciuci}
One has
\begin{equation}\label{ciucu}
REP\circ M_{h(B,v)}^B=\mathbb U^A_v\circ\mathfrak{Op}^A.
\end{equation}

\begin{proof}
For any $f\in L^2(\Xi)$ we have
$$
\mathcal U^A_v\mathfrak{Op}^A(f)\left(\mathcal U^A_v\right)^*\Phi=
\< v^A(X),\mathfrak{Op}^A(f)\left(\mathcal U^A_v\right)^*\Phi\>=
\< v^A(X),\mathfrak{Op}^A(f)\int_\Xi\frac{dY}{(2\pi)^N}\Phi(Y)v^A(Y)\>=
$$
$$
=\int_\Xi\frac{dY}{(2\pi)^N}\Phi(Y)\< v^A(X),\mathfrak{Op}^A(f)v^A(Y)\>=
\int_\Xi\frac{dY}{(2\pi)^N}\Phi(Y)\< v^A,\mathfrak{Op}^A\(\e_{-X}\#^B f\#^B\e_Y\)v^A\>=
$$
$$
=\int_\Xi\frac{dY}{(2\pi)^N}\[M^B_{h(B,v)}(f)\](X,X-Y)\Phi(Y)=\(\[\right(REP\circ M_{h(B,v)}^B\left)(f)\](\Phi)\)(X).
$$
\end{proof}
\end{proposition}

\noindent
{\bf Remark.} We can further compose $M^B_{h(B,v)}$ with a partial Fourier transformation $1\otimes\mathfrak F:
L^2(\Xi\times\Xi)\rightarrow L^2(\Xi\times\Xi)$. Since $OP:=REP\circ(1\otimes\mathfrak F)$ is essentially the Kohn-Nirenberg
pseudodifferential calculus in $\mathbb R^{2N}$, one gets a nice interpretation for the operator $\Op(f)$: its magnetic
Bargmann transform associated to $(A,v)$ is a usual pseudodifferential operator with symbol
$(1\otimes\mathfrak F)\circ M^B_{h(B,v)}$.
This can also be considered a nice interpretation of the magnetic modulation mapping.

%....................................................................................................

\end{document}